\documentclass[12pt]{article}
\usepackage{amsfonts}
\usepackage{amssymb}
\usepackage{amsmath}

\input tcilatex
\begin{document}

\begin{center}
{\Huge Global solution and blow-up for the SKT model in Population Dynamics} 
\\[0pt]

\bigskip 

\textbf{Ichraf Belkhamsa and Messaoud Souilah } \\[0pt]

\bigskip 

Department of Mathematics, Faculty of Sciencs, Univerity Blida 1, P.O.Box
270, Blida, Algeria.

ORCID iD: https://orcid.org/0000-0003-0603-5741.

Department of Analysis, Faculty of Mathematics, University of Science and
Technology Houari Boumediene (USTHB), P.O.Box 32, Algiers, Algeria.

ORCID iD: https://orcid.org/0000-0002-2918-3395.
\end{center}

\textbf{Abstract:} \ 

In this paper, we prove the existence and uniqueness of the global solution
to the reaction diffusion system SKT with homogeneous Newmann boundary
conditions. We use the lower and upper solution method and its associated
monotone iterations, where the reaction functions are locally Lipschitz. We
study the blowing-up property of the solution and give sufficient conditions
on the reaction parameters of the model to ensure the blow-up of the
solution in continuous function spaces.


\section{Introduction}

Various mathematical models from population dynamics translated into
reaction-diffusion systems posed in a bounded domain of $\mathbb{R}^{n}$.
For example, the Sheguesada Kawasaki Teramoto model (see \cite{ref10})
proposed in 1978, which include the following problem:

\begin{equation}
\left \{ 
\begin{array}{l}
u_{1t}-\Delta \lbrack (d_{1}+\alpha _{1}u_{1}+\beta
_{1}u_{2})u_{1}]=u_{1}(a_{1}-b_{1}u_{1}+c_{1}u_{2})\  \text{in }Q_{T} \\ 
u_{2t}-\Delta \lbrack (d_{2}+\alpha _{2}u_{2}+\beta
_{2}u_{1})u_{2}]=u_{2}(a_{2}+b_{2}u_{1}-c_{2}u_{2})\  \text{in }Q_{T} \\ 
\dfrac{\partial u_{1}}{\partial \eta }=\dfrac{\partial u_{2}}{\partial \eta }%
=0\  \text{on }S_{T}=(0,T]\times \partial \Omega \\ 
u_{1}(0,x)=u_{1,0}(x),\ u_{2}(0,x)=u_{2,0}(x)\text{ on }\Omega%
\end{array}%
\right.  \label{39}
\end{equation}
\noindent where $\Omega $ is bounded domain in $\mathbb{R}^{n}(n\geq
1),Q_{T}=(0,T]\times \Omega $$,\ S_{T}$ is the boundary and the closure of $%
\Omega $. $d_{i},\alpha _{i},\beta _{i},a_{i}$ and $b_{i},c_{i}$ are
positive constants, $\Delta=\sum_{i=1}^{n}\partial ^{2}/\partial ^{2}x_{i}$
is the Laplace operator, and $\frac{\partial }{\partial \eta }$ denotes the
directional derivative along the outward normal on $\partial \Omega $. The
problem$(\ref{39})$ has been treated by many researchers, who are devoted to
the blow-up of solution and the global existence with either Newmann or
Dirichle boundary condition by various methods (cf.\cite{ref11,ref19,ref18}%
). In 1984, Kim \cite{ref25} proved the universal existence of classical
solutions to $(\ref{39})$ for $n = 1$ when $d_{1}=d_{2}$, $\alpha_{1}> 0,
\beta_{2} > 0$, and $\beta_{1}= \alpha_{1} = 0$. Recently, Shim \cite{ref26}
refined Kim's findings and established uniform boundaries for solutions by
utilizing a new technique. In several important papers \cite{ref22}, \cite%
{ref23}, and \cite{ref24}, Amann established the local existence of a
solution to (\ref{39}). His results were based on the use of the $W_{p}^{1}$%
-estimates ($p>1$).

If $di=\alpha _{i}=\beta _{i}=0 (i=1,2)$ the problem $(\ref{39}) $ is the
historical Volterra model. For $\alpha _{i}=\beta _{i}=0,$ $(\ref{39})$ is
the Lotka-Volterra system; in this case, C.V.Pao in \cite{ref7} proved that
for $b_{1}c_{2}<c_{1}b_{2}$ the problem admit a unique solution, and for $%
b_{1}c_{2}>c_{1}b_{2}$, the solution blow-up. In the same case, Lou,
Nagylaki and Ni \cite{ref6} studied the effect of diffusion on the blow-up
of solution in finite time $T^{\ast }$. For $n = 2$, Y. Lou, Ni, and Wu
showed in \cite{ref27} that the system (1) has a unique global solution.

On the other hand, Choi, Lui, and Yamada \cite{ref28} examined the problem
when $\alpha_{2} = 0$ and $\beta_{2} = 0$, for any value of $n$. However,
their findings demonstrated the global existence of the solution to problem $%
(\ref{39})$ only when the coefficient $\beta_{1}$ was sufficiently small.
Subsequently, in 2004, they continued their research and established that
the solution to equation $(\ref{39})$ exists globally either when $%
\alpha_{2} = 0$ or when $\alpha_{2} > 0$, with $\beta_{2} = 0$ and $n < 6$
(See \cite{ref29}).

In another approach, Linling.Z, Zhi.Li and Zhigui.Li in \cite{ref20},
devoted to the global existence and blow-up of solutions for system (\ref{39}%
) for $\beta _{i}=0$ with Dirichlet boundary condition, the global solution
is proved by using the upper and lower solution method and sufficient
conditions are given for the solution to blow-up. In particular, for $\beta
_{i}=0,$ R.Hoyer and M.Souilah \cite{ref14} investigated the existence of
global and local solution to the reaction-diffusion system (\ref{39}), they
have given a sufficient condition on the reaction parameters to ensure the
global existence of the solution to the problem in the Sobolev space $%
W^{2,p} $.

In this paper, we study the global and blow-up of the solution to the
reaction diffusion system:

\begin{equation}
\left \{ 
\begin{array}{l}
u_{1t}-\Delta \lbrack (d_{1}+\alpha _{1}u_{1})u_{1}]=f_{1}(u_{1},u_{2})\text{%
\  \ },\text{in }Q_{T}=(0,T]\ast \Omega \\ 
u_{2t}-\Delta \lbrack (d_{2}+\alpha _{2}u_{2})u_{2}]=f_{2}(u_{1},u_{2}\ ),%
\text{in }Q_{T}=(0,T]\ast \Omega \\ 
\dfrac{\partial u_{1}}{\partial \eta }=\dfrac{\partial u_{2}}{\partial \eta }%
=0,\text{on }S_{T}=(0,T]\times \partial \Omega \\ 
u_{1}(0,x)=u_{1,0}(x),\ u_{2}(0,x)=u_{2,0}(x),\text{ on }\Omega%
\end{array}%
\right.  \label{1}
\end{equation}

\noindent where $f_{1}(u_{1},u_{2} )=u_{1}(-a_{1}+b_{1}u_{1}-c_{1}u_{2})$, $%
f_{2}(u_{1},u_{2})=u_{2}(-a_{2}-b_{2}u_{1}+c_{2}u_{2})$,\newline
$f_{1},f_{2}$ are quasimonotone decreasing.

This paper is arranged as follows: In Section 2, we prove the existence and uniqueness of a global solution in continuous spaces by using upper and lower solutions and their associated monotone iterations. Specifically, we create upper and lower sequences that are monotonous of opposing monotony and converge to the same limit, the latter is the unique and global solution to the problem $(\ref{1})$. Thereafter, we give sufficient conditions to ensure the global existence of $(\ref{1})$. 
Section 3 proposes a concept based on creating a lower sequence that reflects the solution to the Cauchy problem, and we illustrate its convergence towards infinity. Finally, we provide sufficient conditions to ensure the explosion.

\section{Existence of global solution}

This section aims to establish the existence of a global solution to $(\ref{1})$. First, we give the definition of ordered upper and lower solutions of $(\ref{1})$:

\begin{definition}
\label{def1}
The pair $w^{(0)}=(w_{1}^{(0)},w_{2}^{(0)}),v^{(0)}=(v_{1}^{(0)},v_{2}^{(0)}) $ in $
C(\Omega )\cap C^{1,2}(\Omega )$ are called order upper and lower solution
of $(\ref{1}),$ if $w^{(0)}\mathbb{\geq }v^{(0)}$ and $w^{(0)},v^{(0)}$
satisfies the following relations:

\begin{equation}
\left \{ 
\begin{array}{l}
(w_{1}^{(0)})_{t}-\Delta \lbrack (d_{1}+\alpha _{1}w_{1}^{(0)})w_{1}^{(0)}]%
\mathbb{\geq }%
w_{1}^{(0)}(-a_{1}+b_{1}w_{1}^{(0)}-c_{1}v_{2}^{(0)})
\\ 
(w_{2}^{(0)})_{t}-\Delta \lbrack (d_{2}+\alpha _{2}w_{2}^{(0)})w_{2}^{(0)}]%
\mathbb{\geq }%
w_{2}^{(0)}(-a_{2}-b_{2}v_{1}^{(0)}+c_{2}w_{2}^{(0)})
\\ 
\dfrac{\partial w_{i}^{(0)}}{\partial \eta }\mathbb{\geq }0 \\ 
w_{i}^{(0)}(x,0)\mathbb{\geq }u_{i,0}(x)%
\end{array}%
\right.  \label{2}
\end{equation}
and
\begin{equation}
\left \{ 
\begin{array}{l}
(v_{1}^{(0)})_{t}-\Delta \lbrack (d_{1}+\alpha _{1}v_{1}^{(0)})v_{1}^{(0)}]%
\mathbb{\leq }%
v_{1}^{(0)}(-a_{1}+b_{1}v_{1}^{(0)}-c_{1}w_{2}^{(0)})
\\ 
(v_{2}^{(0)})_{t}-\Delta \lbrack (d_{2}+\alpha _{2}v_{2}^{(0)})v_{2}^{(0)}]%
\mathbb{\leq }%
v_{2}^{(0)}(-a_{2}-b_{2}w_{1}^{(0)}+c_{2}v_{2}^{(0)})
\\ 
\dfrac{\partial v_{i}^{(0)}}{\partial \eta }\mathbb{\leq }0 \\ 
v_{i}^{(0)}(x,0)\mathbb{\leq }v_{i,0}(x)%
\end{array}%
\right.  \label{3}
\end{equation}
\end{definition}

\noindent Let's define  $J\times J=\{(u_{1},u_{2})\in (C(\overline{Q}))^{2}:(v_{1},v_{2})\mathbb{%
\leq }(u_{1},u_{2})\mathbb{\leq }(w_{1},w_{2})\}$.

\noindent We denote $C(\overline{Q}_{T})=C((0,T]\times \overline{\Omega })$
the space of all bounded and continuous functions in $(0,T]\times \overline{\Omega}$.

Considering only the positive solution of $(\ref{1})$, we define:
\begin{equation}
\begin{array}{l}
P_{i}(u_{i})=h_{i}=\int_{0}^{u_{i}}(d_{i}+2\alpha
_{i}s)ds=(d_{i}+\alpha _{i}u_{i})u_{i}\geq 0 \\ 
h_{it}=(d_{i}+2\alpha _{i}u_{i})u_{it},\  i=1,2  
\end{array}%
\label{101}
\end{equation}
Where
$$
u_{i}=\frac{-d_{i}+\sqrt{d_{i}^{2}+4\alpha _{i}h_{i}}}{2\alpha _{i}}
=q_{i}(h_{i})
$$

\noindent We may express the problem $(\ref{1})$ in the equivalent form
\begin{equation}
\left \{ 
\begin{array}{l}
(d_{i}+2\alpha _{i}u_{i})^{-1}h_{it}-\Delta h_{i}=f_{i}(u_{1},u_{2})\  \text{%
in }Q_{T}\text{ } \\ 
\dfrac{\partial h_{i}}{\partial \eta }=0\text{ on }S_{T} \\ 
h_{i}(x,0)=h_{i,0}(x)\text{ in }\Omega \\ 
u_{i}=q(h_{i})\text{ in }\overline{Q}_{T}%
\end{array}%
\right.  \label{7}
\end{equation}

\noindent Define $Lh_{i}=\Delta h_{i}-\varphi _{i}h_{i}$ where $\varphi _{i}>0$ and 
$F_{i}(u_{1},u_{2})=f_{i}(u_{1},u_{2})+\varphi _{i}h_{i}, \ i=1,2$

\noindent For any $u_{1}>0, u_{2}>0$, the problem $(\ref{7})$  is equivalent to

\begin{equation}
\left \{ 
\begin{array}{l}
(d_{i}+2\alpha _{i}u_{i})^{-1}h_{it}-Lh_{i}=F_{i}(u_{1},u_{2}) \text{in }
Q_{T}\text{ } \\ 
\dfrac{\partial h_{i}}{\partial \eta }=0\text{ on }S_{T} \\ 
h_{i}(x,0)=h_{i,0}(x)\text{ in }\Omega%
\end{array}
\right.  \label{8}
\end{equation}
First, we present the following positivity lemma:
\begin{lemma}
\label{1a}
Let $\sigma (x,t)>0$ in $Q_{T}$, $\beta \geq 0$ on $S_{T}$ and
let either

$(i)$ $e(t,x)>0$ in $Q_{T}$ or

$(ii)$ $(-\dfrac{e}{\sigma (x,t)})$ be bounded in $\overline{Q}_{T}.$

If $z\in C^{1,2}(Q_{T})\cap C(\overline{Q}_{T})$ satisfies the following
inequalities:

$\left \{ 
\begin{array}{l}
\sigma (t,x)z_{t}-\Delta z+e(x,t)z\geq 0\text{ in }Q_{T}\text{ } \\ 
\dfrac{\partial z}{\partial \eta }+\beta (x,t)z\geq 0\text{ in }S_{T} \\ 
z(x,0)\geq 0\text{ in }\Omega%
\end{array}%
\right. $

then $z\geq 0$ in $Q_{T}$
\end{lemma}
\begin{proof}
See proof the lemma 1.2 of  \cite{ref7}.
\end{proof}
\newline
Using either $w^{(0)}$ or $v^{(0)}$ as the initial iteration. We can
construct a sequence $\left( u^{(k)},h^{(k)}\right) $ from the iteration
process for any $k$ and $i=1,2.$

\begin{equation}
\left \{ 
\begin{array}{l}
(d_{i}+2\alpha
_{i}u_{i}^{(k)})^{-1}h_{it}^{(k)}-Lh_{i}^{(k)}=F_{i}(u_{1}^{(k-1)},u_{2}^{(k-1)})
\text{ in }Q_{T}\text{ } \\ 
\dfrac{\partial h_{i}^{(k)}}{\partial \eta }=0\text{ on }S_{T} \\ 
h_{i}^{(k)}(x,0)=h_{i,0}^{(k)}(x)\text{ in }\Omega  \\ 
u_{i}^{(k)}=q(h_{i}^{(k)})\text{ in }\overline{Q}_{T}
\end{array}%
\right.   \label{9}
\end{equation}

\noindent Denote the sequences by $( w^{(k)},\overline{h}^{(k)}) $ and $(v^{(k)},\underline{h}^{(k)})$ respectively, where
$( w^{(0)},\overline{h}^{(0)}) $ and $(v^{(0)},\underline{h}^{(0)})$ are the initial
iteration of the sequences.
\subsection{Monotone sequences}
\label{Subsec:1}

In the following lemma, we show the monotony of the sequences $( w^{(k)},\overline{h}^{(k)}) $ and 
$(v^{(k)},\underline{h}^{(k)}).$
\begin{lemma}
\label{1.2}
The sequences $( w^{(k)},\underline{h}^{(k)}) $ and $(v^{(k)},
\overline{h}^{(k)})$ defined by $(\ref{9})$ possess the monotone property 
\begin{equation}
(v^{(0)},\underline{h}^{(0)})\leq (v^{(k)},\underline{h}^{(k)})\leq
(v^{(k+1)},\underline{h}^{(k+1)})\leq (w^{(k+1)},\overline{h}^{(k+1)})\leq
(w^{(k)},\overline{h}^{(k)})\leq (w^{(0)},\overline{h}^{(0)})  \label{10}
\end{equation}
where $\overline{h}$ and $\underline{h}$ are the upper and lower bounds of $h.$
\end{lemma}
\begin{proof}$\forall k\geq 1,$ where $\overline{h}=P(w)$ and $\underline{h}=P(v)$

\noindent \textbf{Step 1: Initialisation $(k=1)$}
 
\begin{eqnarray*}
\text{Let} \quad  \underline{z}_{i}^{(1)} =\underline{h}_{i}^{(1)}-\underline{h}_{i}^{(0)}%
\text{\ and } \overline{z}_{i}^{(1)} =\overline{h}_{i}^{(0)}-\overline{h}_{i}^{(1)}
\end{eqnarray*}
\noindent for $i=1$, we replace $\underline{z}_{1}^{(1)}$ in $(\ref{9})$
\begin{eqnarray*}
&&(d_{1}+2\alpha _{1}v_{1}^{(1)})^{-1}\underline{z}_{1t}^{(1)}-L_{1}
\underline{z}_{1}^{(1)} \\
&=&F_{1}(v_{1}^{(0)},w_{2}^{(0)})-(d_{1}+2\alpha _{1}v_{1}^{(1)})^{-1}
\underline{h}_{1t}^{(0)}+L_{1}\underline{h}_{1}^{(0)} \\
&=&F_{1}(v_{1}^{(0)},w_{2}^{(0)})+L_{1}\underline{h}_{1}^{(0)}-(d_{1}+2
\alpha _{1}v_{1}^{(0)})^{-1}\underline{h}_{1t}^{(0)}+[(d_{1}+2\alpha
_{1}v_{1}^{(0)})^{-1}-(d_{1}+2\alpha _{1}v_{1}^{(1)})^{-1}]\underline{h}
_{1t}^{(0)} \\
&=&\Delta \underline{h}_{1}^{(0)}+f_{1}(v_{1}^{(0)},w_{2}^{(0)})-(d_{1}+2
\alpha _{1}v_{1}^{(0)})^{-1}\underline{h}_{1t}^{(0)}+[(d_{1}+2\alpha
_{1}v_{1}^{(0)})^{-1}\ -(d_{1}+2\alpha _{1}v_{1}^{(1)})^{-1}]\underline{h}
_{1t}^{(0)}
\end{eqnarray*}

\noindent We use $(\ref{3})$ to get
\begin{equation*}
(d_{1}+2\alpha _{1}v_{1}^{(1)})^{-1}\underline{z}_{1t}^{(1)}-L_{1}\underline{%
z}_{1}^{(1)}\geq -[(d_{1}+2\alpha _{1}v_{1}^{(1)})^{-1}-(d_{1}+2\alpha
_{1}v_{1}^{(0)})^{-1}]\underline{\ h}_{1t}^{(0)}
\end{equation*}

\noindent by the mean value theorem, we obtain
\begin{equation*}
(d_{1}+2\alpha _{1}v_{1}^{(1)})^{-1}-(d_{1}+2\alpha _{1}v_{1}^{(0)})^{-1}=%
\frac{-2\alpha _{1}}{(d_{1}+2\alpha _{1}\xi _{1}^{(0)})^{3}}(\underline{h}%
_{1}^{(1)}-\underline{h}_{1}^{(0)})
\end{equation*}
\noindent for some intermediate value $\xi _{1}^{(0)}$ between $v_{1}^{(0)}$ and $%
v_{1}^{(1)}.$
\noindent we have,
\begin{equation*}
(d_{1}+2\alpha _{1}v_{1}^{(1)})^{-1}\underline{z}_{1t}^{(1)}-L_{1}\underline{%
z}_{1}^{(1)}+\gamma _{1}^{(0)}\underline{z}_{1}^{(1)}\geq 0\text{ in \ }%
\Omega \times (0,T],\text{ \ where }\gamma _{1}^{(0)}=\frac{-2\alpha _{1}}{%
(d_{1}+2\alpha _{1}\xi _{1}^{(0)})^{3}}\underline{h}_{1t}^{(0)}
\end{equation*}
\noindent on the other hand,
\begin{eqnarray*}
\frac{\partial \underline{z}_{1}^{(1)}}{\partial \eta } &=&\frac{\partial (%
\underline{h}_{1}^{(1)}-\underline{h}_{1}^{(0)})}{\partial \eta }=\frac{%
\partial \underline{h}_{1}^{(1)}}{\partial \eta }-\frac{\partial \underline{h%
}_{1}^{(0)}}{\partial \eta }=0\text{ on }\partial \Omega \times (0,T] \\
\underline{z}_{1}^{(1)}(x,0) &=&\  \  \underline{h}_{1}^{(1)}(x,0)-\underline{h%
}_{1}^{(0)}(x,0)\geq 0\text{ in }\Omega
\end{eqnarray*}

\noindent It follows from the lemma  \ref{1a} that $\underline{z}_{1}^{(1)}\geq
0 $ and thus $v_{1}^{(1)}\geq v_{1}^{(0)}.$

 \noindent We replace $\underline{z}_{2}^{(1)}$ in  $(\ref{9})$ and  use a same reasoning we give $\underline{z}_{2}^{(1)}\geq 0$ and thus $v_{2}^{(1)}\geq v_{2}^{(0)}.$

\noindent A similar argument gives $\overline{z}_{i}\geq 0$ and $%
w_{i}^{(0)}\geq w_{i}^{(1)}.$ \newline  Let $y_{i}^{(1)} =\overline{h}_{i}^{(1)}-\underline{h}_{i}^{(1)},\ i=1,2$

\noindent  for $i=1$ we know that  $y_{1}^{(1)}$  satisfies
\begin{eqnarray*}
(d_{1}+2\alpha _{1}w_{1}^{(1)})^{-1}y_{1t}^{(1)}-L_{1}y_{1}^{(1)}+\delta
_{1}^{(1)}y_{1}^{(1)}
=F_{1}(w_{1}^{(0)},v_{2}^{(0)})-F_{1}(v_{1}^{(0)},w_{2}^{(0)})\geq 0
\end{eqnarray*}

\noindent where \ $\delta _{1}^{(1)}=\dfrac{-2\alpha _{1}\underline{h}_{1t}^{(1)}}{%
(d_{1}+2\alpha _{1}\xi _{1}^{(1)})^{3}}$ with $\xi _{1}^{(1)}$ is the
intermediare value between $w_{1}^{(1)}$ and $v_{1}^{(1)}$

\begin{equation*}
\begin{array}{l}
\dfrac{\partial y_{1}^{(1)}}{\partial \eta }=0 \\ 
y_{1}^{(1)}(x,0)=\overline{h}_{1}^{(1)}(x,0)-\underline{h}_{1}^{(1)}(x,0)=0
\end{array}%
\end{equation*}

\noindent It follows again from lemma \ref{1a}, we have $y_{1}^{(1)}\geq 0,$ implies that$\  \overline{h}%
_{1}^{(1)}\geq \underline{h}_{1}^{(1)}$ and therefore $w_{1}^{(1)}\geq
v_{1}^{(1)}.$ The same line of reasoning yields $
y_{2}^{(1)}\geq 0,$ which leads to $\overline{h}_{2}^{(1)}\geq \underline{h}_{2}^{(1)
}$and $w_{2}^{(1)}\geq v_{2}^{(1)\text{ }}.$

\noindent The foregoing conclusion shows that
\begin{equation*}
(v_{i}^{(0)},\underline{h}_{i}^{(0)})\leq (v_{i}^{(1)},\underline{h}%
_{i}^{(1)})\leq (w_{i}^{(1)},\overline{h}_{i}^{(1)})\leq (w_{i}^{(0)},%
\overline{h}_{i}^{(0)})
\end{equation*}

\noindent \textbf{Step 2 : Iteration $(k>1)$}\\
\noindent Assume by induction that
\begin{equation*}
(v_{i}^{(k)},\underline{h}_{i}^{(k)})\leq (v_{i}^{(k+1)},\underline{h}
_{i}^{(k+1)})\leq (w_{i}^{(k+1)},\overline{h}_{i}^{(k+1)})\leq (w_{i}^{(k)},
\overline{h}_{i}^{(k)}), k>1
\end{equation*}
Let  $\underline{z}_{i}^{(k+1)}=\underline{h}_{i}^{(k+1)}-\underline{h}
_{i}^{(k)}$ and $\overline{z}_{i}^{(k+1)}=\overline{h}_{i}^{(k+1)}-\overline{h}_{i}^{(k)}.$\\
Using $(\ref{9})$ again, by the quasi monotone nonincreasing of 
$(f_{1},f_{2})$,  then  $\underline{z}_{1}^{(k+1)}$ satisfies
\begin{equation*}
\left \{ 
\begin{array}{l}
(d_{1}+2\alpha _{1}v_{1}^{(k+1)})^{-1}\underline{z}_{1t}^{(k+1)}-L_{1}%
\underline{z}_{1}^{(k+1)}+\gamma _{1}^{(k)}\underline{z}%
_{1}^{(k+1)}=F_{1}(v_{1}^{(k)},w_{2}^{(k)})-F_{1}(v_{1}^{(k-1)},w_{2}^{(k-1)})\geq 0
\\ 
\dfrac{\partial \underline{z}_{1}^{(k+1)}}{\partial \eta }=0\text{ on }%
\partial \Omega \times (0,T]\text{ } \\ 
\underline{z}_{1}^{(k+1)}(x,0)=0\text{ in }\Omega%
\end{array}%
\right.
\end{equation*}

\noindent where $\gamma _{1}^{(k)}=\dfrac{-2\alpha _{1}}{(d_{1}+2\alpha _{1}\xi
_{1}^{(k)})^{3}}\underline{h}_{1t}^{(k)}, \xi _{1}^{(k)}$ is the intermediare
value between $v_{1}^{(k)}$ and $v_{1}^{(k+1)}$

\noindent by lemma \ref{1a}, we have $\underline{z}_{1}^{(k+1)}\geq 0$ wich leads to $\underline{h}_{1}^{(k+1)}\geq \underline{h}_{1}^{(k)}$ and $v_{1}^{(k+1)}\geq $ $v_{1}^{(k)}.$\\
Similarly, we have  $\underline{z}_{2}^{(k+1)}\geq 0$ impilies that $\underline{h}_{2}^{(k+1)}\geq \underline{h}_{2}^{(k)}$ and $v_{2}^{(k+1)}\geq $ $v_{2}^{(k)}.$
\noindent We use the same procedure to get  $\overline{h}_{i}^{(k+1)}\geq \overline{h}_{i}^{(k)}\geq \underline{h}_{i}^{(k+1)},$ and thus $w_{i}^{(k)}\geq w_{i}^{(k+1)}\geq v_{i}^{(k+1)}$ $(i=1,2).$
\end{proof} 
\subsection{Convergence of the two sequences}

\begin{lemma}( monotony convergence.)
\label{Subsec:2}
From  lemma \ref{1.2} the sequences $(w^{(k)},\overline{h}^{(k)})$ and $\left(v^{(k)},\underline{h}
^{(k)}\right)$ obtained from $(\ref{9})$ are convergent and verify the following limits

\begin{equation*}
\lim_{k\rightarrow +\infty }(w^{(k)},\overline{h}^{(k)})=(w,\overline{h}),%
\text{ }\lim_{k\rightarrow +\infty }(v^{(k)},\underline{h}^{(k)})=(v,
\underline{h})
\end{equation*}
\end{lemma}
In the following, we show that $(w,\overline{h})=(v,\underline{h})=(u^{\ast
},h^{\ast })$ where $u^{\ast }$ is the unique solution for $(\ref{1}).$
\subsection{Uniqueness of the solution}
\label{Subsec:3}

\begin{theorem}
\label {b}
Let $(w_{1}^{(0)},w_{2}^{(0)}),(v_{1}^{(0)},v_{2}^{(0)})$ be the upper and
lower solution of $(\ref{1}).$  Then the sequences $(w^{(k)},\overline{h}^{(k)}
)$ and $(v^{(k)},\underline{h}^{(k)})$ obtained from $(\ref{9})$
converge monotonically to a unique solution $(u^{\ast },h^{\ast })$ of $(\ref%
{7}),$ and they satisfy the relation

$$(v^{(0)},\underline{h}^{(0)})\leq (v^{(k)},\underline{h}^{(k)})\leq
(v^{(k+1)},\underline{h}^{(k+1)})\leq (u^{\ast },h^{\ast })\leq (w^{(k+1)},%
\overline{h}^{(k+1)})\leq (w^{(k)},\overline{h}^{(k)})\leq (w^{(0)},%
\overline{h}^{(0)})\text{,\ }k\geq 1$$
\end{theorem}
The proof of this theorem is similar to Theorem 3.1 in \cite{ref21}.

Creating a pair of ordered upper and lower solutions to $(\ref{1})$ guarantees a global solution to the problem $(\ref{1})$. In fact, we only need to construct an upper solution bounded since the lower solution is sufficiently small and is bounded by $(0,0).$
We have the following theorem

\subsection{Conditions on the parameters of the model}

\label{Subsec:4} 
\begin{theorem}
Let $\lambda _{0}$ be the
first eigenvalue of the Laplacian with 
\begin{center}
$\left \{ 
\begin{array}{l}
2\alpha _{1}\lambda _{0}\leq b_{1} \\ 
2\alpha _{2}\lambda _{0}\leq c_{2} \\ 
(2\alpha _{2}\lambda _{0}-c_{2})(2\alpha _{1}\lambda _{0}-b_{1})-b_{2}c_{1}>0
\end{array}%
\right. $
\end{center}
If the following conditions hold
$$
\frac{-a_{1}(2\alpha _{2}\lambda _{0}-c_{2})+a_{2}c_{1}}{(2\alpha _{2}\lambda _{0}-c_{2})(2\alpha _{1}\lambda
_{0}-b_{1})-b_{2}c_{1}}\leq \frac{a_{1}c_{2}+a_{2}c_{1}}{c_{2}b_{1}-c_{1}b_{2}} 
$$
$$
 \frac{-a_{2}(2\alpha _{1}\lambda _{0}-b_{1})+a_{1}b_{1}}{(2\alpha
_{2}\lambda _{0}-c_{2})(2\alpha _{1}\lambda _{0}-b_{1})-b_{2}c_{1}}\leq \frac{a_{1}b_{2}+a_{2}b_{1}}{%
c_{2}b_{1}-c_{1}b_{2}}
$$

then the problem $(\ref{1})$ admits a unique global solution $(u_{1},u_{2})$
in $\overline{\Omega }\times \lbrack 0,T].$ 

\end{theorem}

\begin{proof}
Let $(w_{1}^{(0)},w_{2}^{(0)}),(v_{1}^{(0)},v_{2}^{(0)})$ the upper and
lower solutions of $(\ref{1})$ with
\begin{equation}
(w_{1}^{(0)},w_{2}^{(0)})=(N_{1},N_{2}),\quad     
(v_{1}^{(0)},v_{2}^{(0)})=(v_{1}^{(0)},v_{2}^{(0)})=(\rho _{1}\Phi _{0},\rho
_{2}\Phi _{0})  \label{32}
\end{equation}%
\ 

where $\Phi _{0}(x)\equiv \Phi _{0}$ is the positive eigenfunction
corresponding to $\lambda _{0}$ satisfying : $\Delta \Phi _{0}=-\lambda
_{0}\Phi _{0}$ and $\dfrac{\partial \Phi _{0}}{\partial \eta }=0.$ 

$\rho _{i},\ N_{i}$ $($ $i=1,2)$ are some positive constant with $\rho _{i}$
sufficiently small.

The pair $(\ref{32})$ satisfy the inequality $(\ref{2}),(\ref{3})$ if :

$\  \  \  \  \  \  \  \  \  \  \  \  \  \  \  \  \  \  \  \  \  \  \  \  \  \  \  \  \  \  \  \  \  \  \  \  \  \
\  \  \  \  \  \  \  \  \  \  \  \  \  \  \  \  \  \  \  \  \  \ $%
\begin{equation}
\left \{ 
\begin{array}{l}
(N_{1})_{t}-\Delta \lbrack (d_{1}+\alpha _{1}N_{1})N_{1}]\mathbf{\geq }%
N_{1}(-a_{1}+b_{1}N_{1}-c_{1}\rho _{2}\Phi _{0}) \\ 
(N_{2})_{t}-\Delta \lbrack (d_{2}+\alpha _{2}N_{2})N_{2}]\mathbf{\geq }%
N_{2}(-a_{2}-b_{2}\rho _{1}\Phi _{0}+c_{2}N_{2}) \\ 
(\rho _{1}\Phi _{0})_{t}-\Delta \lbrack (d_{1}+\alpha _{1}\rho _{1}\Phi
_{0})\rho _{1}\Phi _{0}]\mathbf{\leq }\rho _{1}\Phi _{0}(-a_{1}+b_{1}\rho
_{1}\Phi _{0}-c_{1}N_{2}) \\ 
(\rho _{2}\Phi _{0})_{t}-\Delta \lbrack (d_{2}+\alpha _{2}\rho _{2}\Phi
_{0})\rho _{2}\Phi _{0}]\mathbf{\leq }\rho _{2}\Phi
_{0}(-a_{2}-b_{2}N_{1}+c_{2}\rho _{2}\Phi _{0})%
\end{array}%
\right.  \label{34}
\end{equation}
\newline
\  \ from $\Delta (\Phi _{0}^{2})=2(\Phi _{0}\Delta (\Phi _{0})+\left \vert
\nabla \Phi _{0}\right \vert ^{2})\geq -2\lambda _{0}\Phi _{0}^{2}$ then
\newline
\begin{equation}
\  \left \{ 
\begin{array}{l}
0\mathbf{\geq }-a_{1}+b_{1}N_{1}-c_{1}\rho _{2}\Phi _{0} \\ 
0\mathbf{\geq }-a_{2}-b_{2}\rho _{1}\Phi _{0}+c_{2}N_{2} \\ 
2\alpha _{1}\rho _{1}\lambda _{0}\Phi _{0}\mathbf{\leq }
-a_{1}+b_{1}\rho _{1}\Phi _{0}-c_{1}N_{2}\\ 
2\alpha _{2}\rho _{2}\lambda _{0}\Phi _{0}\mathbf{\leq }-a_{2}-b_{2}N_{1}+c_{2}\rho _{2}\Phi _{0}
\end{array}%
\right.  \label{35}
\end{equation}
\newline
As $N_{i}>\rho _{i}\Phi _{0}>0$ the inequalities $(\ref{35})$ become 

\begin{eqnarray}
-a_{1}+b_{1}N_{1}-c_{1}N_{2} &\leq &0  \label{300} \\
-a_{2}-b_{2}N_{1}+c_{2}N_{2} &\leq &0  \notag
\end{eqnarray}

and

\begin{eqnarray}
a_{1}+(2\alpha _{1}\lambda _{0}-b_{1})N_{1}+c_{1}N_{2} &\mathbf{\leq }&0
\label{37} \\
a_{2}+b_{2}N_{1}+(2\alpha _{2}\lambda _{0}-c_{2})N_{2} &\leq &0  \notag
\end{eqnarray}with
\begin{equation*}
2\alpha _{1}\lambda _{0}-b_{1}\leq 0\text{ and }2\alpha _{2}\lambda
_{0}-c_{2}\leq 0
\end{equation*}%

Assuming that $c_{2}b_{1}>c_{1}b_{2}$,  we get from $(\ref{300})$%
\begin{equation*}
N_{1}\leq \frac{a_{1}c_{2}+a_{2}c_{1}}{c_{2}b_{1}-c_{1}b_{2}}\text{ \  \ and
\  \ } N_{2}\leq \frac{a_{1}b_{2}+a_{2}b_{1}}{c_{2}b_{1}-c_{1}b_{2}}
\end{equation*}

On the outher hand, we find from $(\ref{37})$

\begin{equation*}
N_{1}\geq \frac{-a_{1}(2\alpha _{2}\lambda _{0}-c_{2})+a_{2}c_{1}}{(2\alpha _{2}\lambda _{0}-c_{2})(2\alpha _{1}\lambda
_{0}-b_{1})-b_{2}c_{1}}\text{ and \ }N_{2}\geq \frac{-a_{2}(2\alpha _{1}\lambda _{0}-b_{1})+a_{1}b_{1}}{(2\alpha
_{2}\lambda _{0}-c_{2})(2\alpha _{1}\lambda _{0}-b_{1})-b_{2}c_{1}}\text{\ }
\end{equation*}

where 

\begin{equation*}
(2\alpha _{2}\lambda _{0}-c_{2})(2\alpha _{1}\lambda _{0}-b_{1})-b_{2}c_{1}>0
\end{equation*}
\newline
The relations $(\ref{300})$ and $(\ref{37})$ are satisfied by any constants $%
N_{1},N_{2}$ such that    
$$
\frac{-a_{1}(2\alpha _{2}\lambda _{0}-c_{2})+a_{2}c_{1}}{(2\alpha _{2}\lambda _{0}-c_{2})(2\alpha _{1}\lambda
_{0}-b_{1})-b_{2}c_{1}}\leq N_{1}\leq \frac{a_{1}c_{2}+a_{2}c_{1}}{%
c_{2}b_{1}-c_{1}b_{2}}
$$
$$
 \frac{-a_{2}(2\alpha _{1}\lambda _{0}-b_{1})+a_{1}b_{1}}{(2\alpha
_{2}\lambda _{0}-c_{2})(2\alpha _{1}\lambda _{0}-b_{1})-b_{2}c_{1}}\leq N_{2}\leq \frac{a_{1}b_{2}+a_{2}b_{1}}{%
c_{2}b_{1}-c_{1}b_{2}} 
$$

This suggests that
\begin{equation}
\frac{-a_{1}(2\alpha _{2}\lambda _{0}-c_{2})+a_{2}c_{1}}{(2\alpha _{2}\lambda _{0}-c_{2})(2\alpha _{1}\lambda
_{0}-b_{1})-b_{2}c_{1}}\leq \frac{a_{1}c_{2}+a_{2}c_{1}}{%
c_{2}b_{1}-c_{1}b_{2}} \label{200}
\end{equation}
\begin{equation}
 \frac{-a_{2}(2\alpha _{1}\lambda _{0}-b_{1})+a_{1}b_{1}}{(2\alpha
_{2}\lambda _{0}-c_{2})(2\alpha _{1}\lambda _{0}-b_{1})-b_{2}c_{1}}\leq \frac{a_{1}b_{2}+a_{2}b_{1}}{
c_{2}b_{1}-c_{1}b_{2}}   \label{100}
\end{equation}

Hence under conditions $(\ref{200})$ and $(\ref{100})$, there exist positive constants $ N_{i},\rho _{i}(i=1,2)$  such that the pair in $(\ref{32})$ are coupled upper and lower solutions of  $(\ref{1})$, we conclude from lemma \ref{Subsec:2} that the
sequences $\{w^{(k)},\overline{h}^{(k)}\},\{v^{(k)},\underline{h}^{(k)}\}$
converge monotonically to somme $(w,\overline{h}),(v,\underline{h}).$ It
follows from theorem \ref{b} that $(w,\overline{h})=(v,\underline{h})\equiv
(u^{\ast },h^{\ast })$ is the unique solution of $(\ref{1})$ in $Q_{T}$.
\end{proof}

\section{The Blow-up of the solution}

We announce the following general result on blow-up solutions in the partial
differential equations.

\begin{theorem}
\label {c}
Let $(f_{1},f_{2})$ be locally Lipschitz functions in $\mathbb{R}^{+}\times 
\mathbb{R}^{+}$ and let $(v_{1}^{(0)},v_{2}^{(0)})$ be a positive functions
defined in $\left[ 0,T_{0}\right) \times \overline{\Omega }$ for a finite $%
T_{0}$ and unbounded at some point in $\overline{\Omega }$ as $%
t\longrightarrow T_{0}.$ If $(v_{1}^{(0)},v_{2}^{(0)})$ is a lower solution
of $(\ref{1})$ for every $T<T_{0}$ then there exists $T^{\ast }\mathbf{\leq }%
T_{0}$ such that a unique positive solution $\left( u_{1},u_{2}\right) $ to $%
(\ref{1})$ exists in $\overline{Q}$ and satisfies the relation :%
\begin{equation}
\lim_{t\rightarrow T^{\ast }}\max \left[ u_{1}(t,x)+u_{2}(t,x)\right]
=+\infty   \label{15}
\end{equation}
\end{theorem}
\noindent The proof of theorem \ref{c} is similar to that of theorem 11.5.1 of \cite%
{ref7}.\\
\quad Based on the result of this theorem, the blowing-up property of the solution
to $(\ref{1})$ is ensured if there exists a lower solution unbounded on $\overline{\Omega }$ at a
finite time $ T^{\ast }$. The construction of such a lower solution depends on the type of boundary
conditions. For Neumann boundary conditions, if $(f_{1},f_{2})$ satisfy a growth
condition, the lower solution can often be constructed. We have the growth conditions in the following lemma.
\subsection{Conditions on the reaction terms}
\label{Subsec:5}
\begin{lemma}
\noindent If $(f_{1},f_{2})$ satisfy the local Lipschitz continuous property and
\begin{equation}
\psi _{1}=\min \left \{ \mu _{1}b_{1,}\mu _{2}c_{2}\right \} ,\  \psi
_{2}=\left( \mu _{1}c_{1}+\mu _{2}b_{2}\right) /2,\text{ }c=\max \left \{ \mu
_{1}a_{1},\mu _{2}a_{2}\right \} ,\  \psi _{1}>\psi _{2}>0  \label{30}
\end{equation}

\noindent If $\psi =(\psi _{1}-\psi _{2})/2,$ then $(f_{1},f_{2})$ satisfy the growth condition:
\begin{equation}
\mu _{1}f_{1}(u_{1},u_{2})+\mu _{2}f_{2}(u_{1},u_{2})\geq \psi
(u_{1}+u_{2})^{2}-c(u_{1}+u_{2}),\ u_{1}\geq 0,\ u_{2}\geq 0  \label{17}
\end{equation}

\noindent where $\mu _{i},\ i=1,2$ are positive constants.
\end{lemma}
\begin{proof}
\begin{equation}
\noindent \mu _{1}f_{1}(u_{1},u_{2})+\mu _{2}f_{2}(u_{1},u_{2})=-(\mu
_{1}a_{1}u_{1}+\mu _{2}a_{2}u_{2})+(\mu _{1}b_{1}u_{1}^{2}+\mu
_{2}c_{2}u_{2}^{2})-(\mu _{1}c_{1}+\mu _{2}b_{2})u_{1}u_{2}  \label{18}
\end{equation}

\noindent by $(\ref{30})$, the $(\ref{18})$ becomes

\begin{equation}
\mu _{1}f_{1}(u_{1},u_{2})+\mu _{2}f_{2}(u_{1},u_{2})=\psi
_{1}(u_{1}^{2}+u_{2}^{2})-2\psi _{2}u_{1}u_{2}-c(u_{1}+u_{2})  \label{18a}
\end{equation}

\noindent such that

\begin{equation*}
\psi _{1}(u_{1}^{2}+u_{2}^{2})-2\psi _{2}u_{1}u_{2}=\psi
(u_{1}+u_{2})^{2}-\psi (u_{1}^{2}+u_{2}^{2})+\psi
_{1}(u_{1}^{2}+u_{2}^{2})-2(\psi +\psi _{1})u_{1}u_{2}
\end{equation*}
We get
\begin{equation*}
\psi _{1}(u_{1}^{2}+u_{2}^{2})-2\psi _{2}u_{1}u_{2}\geq \psi
(u_{1}+u_{2})^{2}+2(\psi _{1}-2\psi -\psi _{2})u_{1}u_{2}
\end{equation*}

\noindent

So,
\begin{equation*}
\psi _{1}(u_{1}^{2}+u_{2}^{2})-2\psi _{2}u_{1}u_{2}\geq \psi
(u_{1}+u_{2})^{2}.
\end{equation*}

\noindent The relation $(\ref{18a})$ becomes
\begin{equation*}
\mu _{1}f_{1}(u_{1},u_{2})+\mu _{2}f_{2}(u_{1},u_{2})\geq \psi
(u_{1}+u_{2})^{2}-c(u_{1}+u_{2}),\ u_{1}\geq 0,\ u_{2}\geq 0
\end{equation*}
\end{proof}

The lower solution can be constructed by using the solution of the Cauchy problem:
\begin{equation*}
\left \{ 
\begin{array}{l}
\widehat{p}^{\prime }+\overline{\tau }\widehat{p}=\underline{\psi }\widehat{p%
}^{2} \\ 
\widehat{p}(0)=\widehat{p}_{0}%
\end{array}%
\right. 
\end{equation*}%
 where $\overline{\tau }$,$\underline{\psi }$
are arbitrary non-zero constants. In the following, we show the  blow-up of the problem $(\ref{1})$
\begin{theorem}
\label {d}
Let $(f_{1},f_{2})$ be locally Lipschitz function in $\mathbb{R}^{+}\times 
\mathbb{R}^{+}$ and satisfy the condition $(\ref{30})$ and the relation $(%
\ref{17})$, and let $\left( u_{1},u_{2}\right) $ be the local positive
solution of $(\ref{1})$. Then for any nonnegative $\left( u_{1}(0,x),u_{2}(0,x)\right) $ satisfying the condition
\begin{equation}
\widehat{p}_{0}>\dfrac{\overline{\tau }}{\underline{\psi }}  \label{19}
\end{equation}
 there exists a finite $T^{\ast }\leq T_{0}$ such that $\left(
u_{1},u_{2}\right) $ possesses the blowing-up property $(\ref{15}),$ where $%
T_{0}$\begin{equation}
T_{0}=\dfrac{1}{\overline{\tau }}\ln [\dfrac{\underline{\psi }\widehat p_{0}}{%
\underline{\psi }\widehat p_{0}-\overline{\tau }}].  \label{T0}
\end{equation}
\end{theorem}

\begin{proof}
Consider $(v_{1}^{(0)},v_{2}^{(0)})=(\widehat{p}_{1}(t)\Phi _{0},\widehat{p}_{2}(t)\Phi _{0})$ is a lower solution of $(\ref{1})$ where $\Phi _{0}$  is the positive eigenfunction corresponding to the eigenvalue $
\lambda _{0}$.  Let $\left( u_{1},u_{2}\right) $ be the local positive
solution of $(\ref{1})$ and define the weighted average of the solution $\left( u_{1},u_{2}\right)$ by
 
\begin{equation*}
\widehat{p}_{1}(t)=\left \vert \Omega \right \vert ^{-1}\int \nolimits_{\Omega
}\Phi _{0}(x)u_{1}(x,t)dx,\text{ \  \ }\widehat{p}_{2}(t)=\left \vert \Omega
\right \vert ^{-1}\int \nolimits_{\Omega }\Phi _{0}(x)u_{2}(x,t)dx\text{\ and}
\end{equation*}%
\begin{equation}
\widehat{p}(t)=\left \vert \Omega \right \vert ^{-1}\int \nolimits_{\Omega
}\Phi _{0}(x)[\mu _{1}u_{1}(x,t)+\mu _{2}u_{2}(x,t)]dx=\mu _{1}\widehat{p}%
_{1}(t)+\mu _{2}\widehat{p}_{2}(t)\quad\text{with}\quad \widehat{p}(0)=\widehat{p}_{0}  \label{21}
\end{equation}
\noindent Multiply the first equation in $(\ref{1})$ by $\Phi
_{0}$ of the laplacian and integrate over $\Omega $ using Green's
theorem which implies

\begin{equation*}
\int_{\Omega }u_{1t}\Phi _{0}dx\geq \int_{\Omega }\Delta \Phi
_{0}[(d_{1}+\alpha _{1}u_{1})u_{1}]dx+\int_{\Omega }f_{1}(u_{1},u_{2})\Phi
_{0}dx
\end{equation*}

\begin{equation}
\int_{\Omega }u_{1t}\Phi _{0}dx\geq -\int_{\Omega }\lambda _{0}\Phi
_{0}d_{1}u_{1}(x,t)dx-\int_{\Omega }\lambda _{0}\alpha
_{1}u_{1}^{2}(x,t)\Phi _{0}dx+\int_{\Omega }f_{1}(u_{1},u_{2})\Phi _{0}dx
\label{22}
\end{equation}

\noindent Multiply $(\ref{22})$ by\ $\mu _{1}$
\begin{equation}
\mu _{1}\left \vert \Omega \right \vert \widehat{p}_{1}^{\prime }(t)\geq
-\mu _{1}\lambda _{0}d_{1}\left \vert \Omega \right \vert \widehat{p}%
_{1}(t)-\lambda _{0}\alpha _{1}\mu _{1}\left \vert \Omega \right \vert 
\widehat{p}_{1}^{2}(t)+\int_{\Omega }\mu _{1}f_{1}(u_{1},u_{2})\Phi _{0}dx
\label{23}
\end{equation}

\noindent An analogue reasoning for the second equation in $(\ref{1})$, gives

\begin{equation}
\mu _{2}\left \vert \Omega \right \vert \widehat{p}_{2}^{\prime }(t)\geq
-\mu _{2}\lambda _{0}d_{2}\left \vert \Omega \right \vert \widehat{p}%
_{2}(t)-\lambda _{0}\alpha _{2}\mu _{2}\left \vert \Omega \right \vert 
\widehat{p}_{2}^{2}(t)+\int_{\Omega }\mu _{2}f_{2}(u_{1},u_{2})\Phi _{0}dx
\label{24}
\end{equation}

\noindent The combination of the above inequalities $(\ref{23})$ and $(\ref{24})$ gives
$$
\left \vert \Omega \right \vert [\mu _{1}\widehat{p}_{1}^{\prime }(t)+\mu_{2}\widehat{p}_{2}^{\prime }(t)] \geq -
\lambda _{0}\left \vert \Omega\right \vert [\mu _{1}d_{1}\widehat{p}_{1}(t)+\mu _{2}d_{2}\widehat{p}_{2}(t)]-$$
\begin{equation}
  \lambda _{0}\left \vert \Omega \right \vert [\mu _{1}\alpha _{1}
\widehat{p}_{1}^{2}(t)+\mu _{2}\alpha _{2}\widehat{p}_{2}^{2}(t)] 
  +\int_{\Omega }[\mu _{1}f_{1}(u_{1},u_{2})+\mu _{2}f_{2}(u_{1},u_{2})]\Phi
_{0}dx  \label{25}    
 \end{equation}
\noindent We notice that
\begin{equation*}
\widehat{p}^{2}(t)>\mu _{1}^{2}\widehat{p}_{1}^{2}(t)\text{ and }\widehat{p}%
^{2}(t)>\mu _{2}^{2}\widehat{p}_{2}^{2}(t)
\end{equation*}

\noindent If $\lambda _{0}>0$ multiply  the two inequalities by $-\lambda _{0}\left \vert \Omega
\right
\vert \alpha _{1},\ -\lambda _{0}\left \vert \Omega \right \vert
\alpha _{2}$ respectively, we get

\begin{equation*}
-\lambda _{0}\left \vert \Omega \right \vert \alpha _{1}\widehat{p}%
^{2}(t)<-\lambda _{0}\left \vert \Omega \right \vert \mu _{1}^{2}\alpha _{1}%
\widehat{p}_{1}^{2}(t)\text{ and }-\lambda _{0}\left \vert \Omega \right
\vert \alpha _{2}\widehat{p}^{2}(t)<-\lambda _{0}\left \vert \Omega \right
\vert \mu _{2}^{2}\alpha _{2}\widehat{p}_{2}^{2}(t)
\end{equation*}

\noindent taking\ $\alpha =\alpha _{1}+\alpha _{2},$ $\overline{\mu }=\max \left \{
\mu _{1},\mu _{2}\right \} $ and substitute in the above inequalities
\begin{equation}
-\lambda _{0}\left \vert \Omega \right \vert \alpha \widehat{p}%
^{2}(t)<-\lambda _{0}\left \vert \Omega \right \vert \overline{\mu }[\mu
_{1}\alpha _{1}\widehat{p}_{1}^{2}(t)+\mu _{2}\alpha _{2}\widehat{p}%
_{2}^{2}(t)] \label{26}
\end{equation}

\noindent for $\overline{d}=\max \left \{ d_{1},d_{2}\right \}$, then by $(\ref{26})$ and the growth condition $(
\ref{17})$ on $(f_{1},f_{2})$, the relation $(\ref{25})$  becomes

\begin{equation}
\left \vert \Omega \right \vert \widehat{p}^{\prime }(t)+\lambda _{0}\left
\vert \Omega \right \vert \overline{d}\widehat{p}(t)+\lambda _{0}\left \vert
\Omega \right \vert \alpha \overline{\mu }^{-1}\widehat{p}%
^{2}(t)>\int_{\Omega }\psi (u_{1}+u_{2})^{2}\Phi _{0}dx-\int_{\Omega
}c(u_{1}+u_{2})\Phi _{0}dx \label{27}
\end{equation}

\noindent We use the relation:
\begin{equation*}
\overline{\mu }^{-1}(\mu _{1}u_{1}+\mu _{2}u_{2})\leq (u_{1}+u_{2})\leq 
\underline{\mu }^{-1}(\mu _{1}u_{1}+\mu _{2}u_{2}),\text{ }\underline{\mu }%
=\min (\mu _{1},\mu _{2})
\end{equation*}
We get
\begin{equation}
-\int_{\Omega }c(u_{1}+u_{2})\Phi _{0}dx\geq -\int_{\Omega }c\underline{\mu }%
^{-1}(\mu _{1}u_{1}+\mu _{2}u_{2})\Phi _{0}dx\geq -c\underline{\mu }%
^{-1}\left \vert \Omega \right \vert \widehat{p}(t)  \label{27a}
\end{equation}
and
\begin{equation}
\int_{\Omega }\psi (u_{1}+u_{2})^{2}\Phi _{0}dx\geq \psi \overline{\mu }%
^{-2}\int_{\Omega }(\mu _{1}u_{1}+\mu _{2}u_{2})^{2}\Phi _{0}dx  \label{27b}
\end{equation}
Since by Holder's inequality,
$$\int_{\Omega }\Phi _{0}(\mu _{1}u_{1}+\mu _{2}u_{2})dx\leq (\int_{\Omega
}dx)^{1/2}.(\int_{\Omega }[\Phi _{0}(\mu _{1}u_{1}+\mu
_{2}u_{2})]^{2}dx)^{1/2}$$

Relation ($\ref{21}$) and $\Phi _{0}\leq 1$ imply that

\begin{equation}
 \int_{\Omega }(\mu _{1}u_{1}+\mu _{2}u_{2})^{2}\Phi _{0}dx\geq
\left \vert \Omega \right \vert \widehat{p}^{2}(t)  \label{27c}
\end{equation}
 The inequality $(\ref{27b})$ leads to
\begin{equation}
\int_{\Omega }\psi (u_{1}+u_{2})^{2}\Phi _{0}dx\geq\psi \overline{\mu }
^{-2}
\left \vert \Omega \right \vert \widehat{p}^{2}(t)  \label{27e}
\end{equation}
\noindent We replace $(\ref{27a})$ and $(\ref{27e})$ in the inequality $(\ref{27})$, we obtain
\begin{equation}
\widehat{p}^{\prime }(t)+(\lambda _{0}\overline{d}+c\underline{\mu }^{-1})%
\widehat{p}(t)\geq (\psi \overline{\mu }^{-2}-\lambda _{0}\alpha \overline{\mu }^{-1})\widehat{p}^{2}(t)  \label{28}
\end{equation}

\noindent taking $\overline{\tau }=\lambda _{0}\overline{d}+c\underline{\mu }^{-1};\ 
\underline{\psi }=\psi \overline{\mu }^{-2}-\lambda _{0}\alpha \overline{\mu }^{-1}$, this gives%
\begin{equation*}
\widehat{p}^{\prime }(t)+\overline{\tau }\widehat{p}(t)\geq \underline{\psi }%
\widehat{p}^{2}(t)
\end{equation*}

\noindent Integrate the above inequality to get

\begin{equation}
\widehat{p}(t)\geq \dfrac{\exp (-\overline{\tau }t)}{\dfrac{1}{\widehat{p}%
_{0}}-\dfrac{\underline{\psi }}{\overline{\tau }}(1-\exp (-\overline{\tau }%
t))},0\leq t\leq T_{0}  \label{29}
\end{equation}
\noindent With this function $\widehat{p}(t)$, $(v_{1}^{(0)},v_{2}^{(0)})=(\widehat{p}_{1}(t)\Phi _{0},\widehat{p}_{2}(t)\Phi _{0})$ is a lower solution of $(\ref{1})$ for every $T<T_{0}$  where $T_{0}$ is given by the right-hand side of $(\ref{T0}).$
\noindent Since for $\widehat{p}_{0}>\overline{\tau }/\underline{\psi }$ the function $
\widehat{p}(t)$ grows unbounded as $t\rightarrow T_{0}$, there exists 
$T^{\ast }\leq T_{0}$ such that \ $\widehat{p}(t)\rightarrow +\infty $ as $
t\rightarrow T^{\ast }$.  By relation $(\ref{21})$, $(u_{1},u_{2})$ must be unbounded in $\left[0,T^{\ast }\right] \times \overline{\Omega }$, which implies that the unique positive solution
$\left(u_{1},u_{2}\right)$ possess the blowing-up property $(\ref{15})$.
\end{proof}
\subsection{Discussion on the blow-up parameters}

\label{Subsec:6} Following the notations in (\ref{30}), we obtain the two
conditions on the parameters causing blow-up of the solution

\begin{itemize}
\item If $0<\mu _{1}<\mu _{2}$ we choose $\psi _{1}=\mu _{1}b_{1}$ then%
\begin{equation*}
\mu _{1}b_{2}<\mu _{2}b_{2}\Leftrightarrow \frac{\mu _{1}c_{1}+\mu _{1}b_{2}%
}{2}<\frac{\mu _{1}c_{1}+\mu _{2}b_{2}}{2}
\end{equation*}
\end{itemize}

\begin{equation*}
\frac{\mu _{1}c_{1}+\mu _{1}b_{2}}{2}<\psi _{2}<\psi _{1}
\end{equation*}

\noindent we obtain the first condition 
\begin{equation}
c_{1}+b_{2}<2b_{1}
\end{equation}

\begin{itemize}
\item If $\mu _{1}>\mu _{2}>0,$ we choose $\psi _{1}=\mu _{2}c_{2}$ then
\end{itemize}

\begin{equation*}
\frac{\mu _{1}c_{1}+\mu _{2}b_{2}}{2}>\frac{\mu _{2}c_{1}+\mu _{2}b_{2}}{2}
\end{equation*}

\begin{equation*}
\psi _{1}>\frac{\mu _{1}c_{1}+\mu _{2}b_{2}}{2}>\frac{\mu _{2}c_{1}+\mu
_{2}b_{2}}{2}
\end{equation*}

\begin{equation*}
\mu _{2}c_{2}>\frac{\mu _{2}c_{1}+\mu _{2}b_{2}}{2}
\end{equation*}

\noindent we obtain the second condition : 
\begin{equation}
2c_{2}>c_{1}+b_{2}
\end{equation}

%

\section{Conclusions}

We have obtained the following conditions for the golbal existence of the
solution to the SKT problem :

\noindent let $2\alpha _{1}\lambda _{0}-b_{1}\leq 0$, $2\alpha _{2}\lambda
_{0}-c_{2}\leq 0$ and $(2\alpha _{2}\lambda _{0}-c_{2})(2\alpha _{1}\lambda
_{0}-b_{1})-b_{2}c_{1}>0$ with $c_{2}b_{1}>c_{1}b_{2},$ If the following
conditions hold 
\begin{equation*}
\frac{-a_{1}(2\alpha _{2}\lambda _{0}-c_{2})(2\alpha _{1}\lambda
_{0}-b_{1})+a_{2}c_{1}}{(2\alpha _{2}\lambda _{0}-c_{2})(2\alpha _{1}\lambda
_{0}-b_{1})-b_{2}c_{1}}\leq \frac{a_{1}c_{1}+a_{2}c_{2}}{%
c_{2}b_{1}-c_{1}b_{2}}
\end{equation*}%
\begin{equation}
\frac{-a_{2}(2\alpha _{2}\lambda _{0}-c_{2})(2\alpha _{1}\lambda
_{0}-b_{1})+a_{1}b_{1}}{(2\alpha _{2}\lambda _{0}-c_{2})(2\alpha _{1}\lambda
_{0}-b_{1})-b_{2}c_{1}}\leq \frac{a_{1}b_{2}+a_{2}b_{1}}{%
c_{2}b_{1}-c_{1}b_{2}}
\end{equation}%
then the problem $(\ref{1})$ admits a unique global solution $(u_{1},u_{2})$
in $\overline{\Omega }\times \lbrack 0,+\infty ).$

\noindent For the blow-up conditions parameters, we have obtained :

\noindent if $\psi _{1}=\mu _{1}b_{1},\  \psi _{2}=\dfrac{\left( \mu
_{1}c_{1}+\mu _{2}b_{2}\right) }{2}$ with $\psi _{2}<\psi _{1}$ and $0<\mu
_{1}<\mu _{2}$ then $c_{1}+b_{2}<2b_{1}.$

\noindent if $\psi _{1}=\mu _{2}c_{2},\  \psi _{2}=\dfrac{\left( \mu
_{1}c_{1}+\mu _{2}b_{2}\right) }{2}$ with $\psi _{2}<\psi _{1}$ and $\mu
_{1} $ $>\mu _{2}>0$ then $c_{1}+b_{2}<2c_{2}.$

\noindent Under these conditions the solution of problem $(\ref{1})$ blows
up as $t\rightarrow T^{\ast }.$


\end{document}